\documentclass[11pt, reqno]{amsart}

\usepackage{mathrsfs}
\usepackage{amsfonts}
\usepackage{amsthm}
\usepackage{amsmath}
\usepackage{url}

\newtheorem{thm}{Theorem}[section]
\newtheorem{lemma}[thm]{Lemma}
\newtheorem{cor}[thm]{Corollary}

\theoremstyle{definition}
\newtheorem{definition}[thm]{Definition}
\newtheorem{ex}[thm]{Example}

\theoremstyle{remark}
\newtheorem{remark}[thm]{Remark}
\newtheorem{extProblem}{Extremal Problem}[section]

\numberwithin{equation}{section}

\DeclareMathOperator{\rank}{rank}
\DeclareMathOperator{\Range}{Range}
\DeclareMathOperator{\ess}{ess}
\DeclareMathOperator{\trace}{trace}
\DeclareMathOperator{\dist}{dist}
\DeclareMathOperator{\ind}{ind}
\DeclareMathOperator{\codim}{codim}

\newcommand{\cA}{\mathcal{A}}
\newcommand{\cB}{\mathcal{B}}
\newcommand{\cF}{\mathcal{F}}

\newcommand{\cS}{\sigma}

\newcommand{\cU}{\mathcal{U}}

\newcommand{\sL}{\mathscr{L}}

\newcommand{\bS}{\textbf{\textsl{S}}}

\newcommand{\C}{\mathbb{C}}
\newcommand{\D}{\mathbb{D}}
\newcommand{\M}{\mathbb{M}}
\newcommand{\N}{\mathbb{N}}
\newcommand{\T}{\mathbb{T}}
\newcommand{\pP}{\mathbb{P}}

\newcommand{\dm}{d{\boldsymbol{m}}}
\newcommand{\clos}{{\rm clos}}
\newcommand{\Zero}{\mathbb{O}}
\newcommand{\Mydef}{\stackrel{\mbox{\footnotesize{\rm def}}}{=}}

\begin{document}
\title{On the sum of superoptimal singular values}

\author{Alberto A. Condori}
\address{Department of Mathematics, Michigan State University, East Lansing, MI 48824, USA}
\email{condoria@msu.edu}
 
\subjclass[2000]{Primary 47B35; Secondary 46E40, 47L20}
\keywords{Hankel operators, Toeplitz operators, best approximation, badly approximable matrix functions, 
superoptimal approximation, maximizing vector.}

\begin{abstract}
			In this paper, we study the following extremal problem and its relevance to the sum of the so-called 
			superoptimal singular values of a matrix function: Given an $m\times n$ matrix function $\Phi$, when 
			is there a matrix function $\Psi_{*}$ in the set $\cA_{k}^{n,m}$ such that
			\[	\int_{\T}\trace(\Phi(\zeta)\Psi_{*}(\zeta))\dm(\zeta)
					=\sup_{\Psi\in\cA_{k}^{n,m}}\left|	\int_{\T}\trace(\Phi(\zeta)\Psi(\zeta))\dm(\zeta)\right|?	\]
			The set $\cA_{k}^{n,m}$ is defined by
			\[	\cA_{k}^{n,m}\Mydef\left\{	\Psi\in H_{0}^{1}(\M_{n,m}): \|\Psi\|_{L^{1}(\M_{n,m})}\leq 1,
									\,\rank \Psi(\zeta)\leq k\;\mbox{ a.e. }\zeta\in\T	\right\}.	\]
			To address this extremal problem, we introduce Hankel-type operators on spaces of matrix functions and 
			prove that this problem has a solution if and only if the corresponding Hankel-type operator has a 
			maximizing vector.  The main result of this paper is a characterization of the smallest number $k$ for which
			\[	\int_{\T}\trace(\Phi(\zeta)\Psi(\zeta))\dm(\zeta)	\]
			equals the sum of all the superoptimal singular values of an admissible matrix function $\Phi$ (e.g. a 
			continuous matrix function) for some function $\Psi\in\cA_{k}^{n,m}$.  Moreover, we provide a representation 
			of any such function $\Psi$ when $\Phi$ is an admissible very badly approximable unitary-valued 
			$n\times n$ matrix function.
\end{abstract}

\maketitle

{\centering \section{Introduction}}

The problem of best analytic approximation for a given $m\times n$ matrix-valued bounded function $\Phi$ 
on the unit circle $\T$ is to find a bounded analytic function $Q$ such that
\[	\|\Phi-Q\|_{L^{\infty}(\M_{m,n})}=\inf\{\|\Phi-F\|_{L^{\infty}(\M_{m,n})}: F\in H^{\infty}(\M_{m,n})\}.	\]
Throughout, 
\[	\|\Psi\|_{L^{\infty}(\M_{m,n})}\Mydef\ess\sup_{\zeta\in\T}\|\Psi(\zeta)\|_{\M_{m,n}},	\]
$\M_{m,n}$ denotes the space of $m\times n$ matrices equipped with the operator norm $\|\cdot\|_{\M_{m,n}}$
(of the space of linear operators from $\C^{n}$ to $\C^{m}$), and $H^{\infty}(\M_{m,n})$ denotes the space of 
bounded analytic $m\times n$ matrix-valued functions on $\T$.

It is well-known that, unlike scalar-valued functions, a polynomial matrix function $\Phi$ may have many best 
analytic approximants.  Therefore it is natural to impose additional conditions in order to distinguish a  
``very best'' analytic approximant among all best analytic approximants.  To do so here, we use the notion of 
superoptimal approximation by bounded analytic matrix functions.

\smallskip

\subsection{Superoptimal approximation and very badly approximable matrix functions}\label{superApproxSection}

Recall that for an $m\times n$ matrix $A$, the $j$th-\emph{singular value} $s_{j}(A)$, $j\geq 0$, is defined to be 
the distance from $A$ to the set of matrices of rank at most $j$ under the operator norm.  More precisely,
\[	s_{j}(A)=\inf\{\|A-B\|_{\M_{m,n}}: B\in\M_{m,n}\mbox{ such that }\rank B\leq j \}.	\]
Clearly, $s_{0}(A)=\|A\|_{\M_{m,n}}$.

\begin{definition}
			Let $\Phi\in L^{\infty}(\M_{m,n})$.  For $k\geq0$, we define the sets $\Omega_{k}=\Omega_{k}(\Phi)$ by
			\begin{align*}
						\Omega_{0}(\Phi)&=\left\{	F\in H^{\infty}(\M_{m,n}): F \mbox{ minimizes }\ess\sup_{\zeta\in\T}\|\Phi(\zeta)-F(\zeta)\|_{\M_{m,n}}\right\},
																\mbox{ and}\\
						\Omega_{j}(\Phi)&=\left\{	F\in\Omega_{j-1}: F \mbox{ minimizes }\ess\sup_{\zeta\in\T}s_{j}(\Phi(\zeta)-F(\zeta))	\right\}\mbox{ for }j>0.
			\end{align*}
			Any function $\displaystyle{F\in\bigcap_{k\geq0}\Omega_{k}=\Omega_{\min\{m,n\}-1}}$ is called a \emph{superoptimal 
			approximation to $\Phi$} by bounded analytic matrix functions.  In this case, the \emph{superoptimal singular 
			values of $\Phi$} are defined by 
			\[	t_{j}=t_{j}(\Phi)=\ess\sup_{\zeta\in\T}s_{j}((\Phi-F)(\zeta))\mbox{ for }j\geq 0.	\]
			Moreover, if the zero matrix function $\Zero$ belongs to $\Omega_{\min\{m,n\}-1}$, we say that \emph{$\Phi$ 
			is very badly approximable}.			
\end{definition}

Notice that any function $F\in \Omega_{0}$ is a best analytic approximation to $\Phi$.  Also, any very badly
approximable matrix function is the difference between a bounded matrix function and its superoptimal approximant.

It turns out that Hankel operators on Hardy spaces play an important role in the study of superoptimal approximation.
For a matrix function $\Phi\in L^{\infty}(\M_{m,n})$, we define the \emph{Hankel operator} $H_{\Phi}$ by
\[	H_{\Phi}f=\pP_{-}\Phi f,\,\mbox{ for }f\in H^{2}(\C^{n}),	\]
where $\pP_{-}$ denotes the orthogonal projection from $L^{2}(\C^{m})$ onto $H^{2}_{-}(\C^{m})\Mydef
L^{2}(\C^{m})\ominus H^{2}(\C^{m})$.

When studying superoptimal approximation, we only consider bounded matrix functions that are admissible.  A matrix 
function $\Phi\in L^{\infty}(\M_{m,n})$ is said to be \emph{admissible} if the essential norm $\|H_{\Phi}\|_{\rm e}$ of 
the Hankel operator $H_{\Phi}$ is strictly less than the smallest non-zero superoptimal singular value of $\Phi$.  As usual, 
the essential norm of a bounded linear operator $T$ between Hilbert spaces is defined by
\[	\|T\|_{\rm e}\Mydef\{ \|T-K\|: K\emph{ is compact }\}.	\]
Note that any continuous matrix function $\Phi$ is admissible, as the essential norm of $H_{\Phi}$ equals zero in this 
case.  Moreover, in the case of scalar-valued functions, to say that a function $\varphi$ is admissible simply means
that $\|H_{\varphi}\|_{\rm e}<\|H_{\varphi}\|$.

It is known that if $\Phi$ is an admissible matrix function, then $\Phi$ has a unique superoptimal approximation 
$Q$ by bounded analytic matrix functions.  Moreover, the functions $\zeta\mapsto s_{j}((\Phi-Q)(\zeta))$ equal
$t_{j}(\Phi)$ a.e. on $\T$ for each $j\geq 0$. These results were first proved in \cite{PY} for the special case
$\Phi\in(H^{\infty}+C)(\M_{m,n})$ (i.e. matrix functions which are a sum of a bounded analytic matrix function and 
a continuous matrix function), and shortly after proved for the class of admissible matrix functions in \cite{PT}.  

While it is possible to compute the superoptimal singular values of a given matrix function in concrete examples, it 
is not known how to verify if a matrix function that is not continuous is admissible or not.  Thus a complete
characterization of the smallest non-zero superoptimal singular value of a given matrix function is an important
problem for superoptimal approximation.  This remains an open problem.

We refer the reader to Chapter 14 of \cite{Pe1} which contains proofs to all of the previously mentioned results and 
many other interesting results concerning superoptimal approximation.\smallskip

\subsection{An extremal problem}

Throughout this note, $\boldsymbol{m}$ denotes normalized Lebesgue measure on $\T$ so that $\boldsymbol{m}(\T)=1$.

\begin{definition}
			Let $m,n>1$ and $1\leq k\leq\min\{m,n\}$.  For $\Phi\in L^{\infty}(\M_{m,n})$, we define $\cS_{k}(\Phi)$ by
			\begin{equation}\label{supQuantity}
						\cS_{k}(\Phi)\Mydef\sup_{\Psi\in \cA_{k}^{n,m}}\left|\int_{\T}\trace(\Phi(\zeta)\Psi(\zeta))\dm(\zeta)\right|,
			\end{equation}
			where
			\[	\cA_{k}^{n,m}=\left\{	\Psi\in H_{0}^{1}(\M_{n,m}): \|\Psi\|_{H_{0}^{1}(\M_{n,m})}\leq 1\,\mbox{ and }
										\;\rank \Psi(\zeta)\leq k\;\mbox{ a.e. }\zeta\in\T	\right\}.	\]
			Whenever $n=m$, we use the notation $\cA_{k}^{n}\Mydef\cA_{k}^{n,m}$.
\end{definition}

We are interested in the following extremal problem:\smallskip

\begin{extProblem}\label{extremal}
\emph{For a matrix function $\Phi\in L^{\infty}(\M_{m,n})$, when is there a matrix function $\Psi\in\cA_{k}^{n,m}$ 
such that}
\[	\int_{\T}\trace(\Phi(\zeta)\Psi(\zeta))\dm(\zeta)=\cS_{k}(\Phi)?	\]
\end{extProblem}

The importance of this problem arose from the following observation due to Peller \cite{Pe3}.

\begin{thm}\label{SInequality}
			Let $1\leq k\leq\min\{m,n\}$. If $\Phi\in L^{\infty}(\M_{m,n})$ is admissible, then
			\begin{equation}\label{traceIneq}
						\cS_{k}(\Phi)\leq	 t_{0}(\Phi)+\ldots+t_{k-1}(\Phi).
			\end{equation}
\end{thm}
\begin{proof}
			Let $\Psi\in\cA_{k}^{n,m}$.  We may assume, without loss of generality, that $\Phi$ is very badly approximable.
			Indeed,
			\[	\int_{\T}\trace(\Phi(\zeta)\Psi(\zeta))\dm(\zeta)=\int_{\T}\trace((\Phi-Q)(\zeta)\Psi(\zeta))\dm(\zeta)	\]
			holds for any $Q\in H^{\infty}(\M_{m,n})$, and so we may replace $\Phi$ with $\Phi-Q$ if necessary, where $Q$ 
			is the superoptimal approximation to $\Phi$ in $H^{\infty}(\M_{m,n})$.
			
			Let $\bS_{1}^{m}$ denote the collection of $m\times m$ matrices equipped with the \emph{trace norm}
			$\|A\|_{\bS_{1}^{m}}=\trace(A^{*}A)^{1/2}=\sum_{j\geq 0}s_{j}(A)$.  
			
			It follows from the well-known identity $|\trace(A)|\leq\|A\|_{\bS_{1}^{m}}$ that the inequalities
			\[	|\trace(\Phi(\zeta)\Psi(\zeta))|\leq\|\Phi(\zeta)\Psi(\zeta)\|_{\bS_{1}^{m}}\leq\left(\sum_{j=0}^{k-1}s_{j}(\Phi(\zeta))\right)
					\|\Psi(\zeta)\|_{\M_{n,m}}	\]
			hold for a.e. $\zeta\in\T$.  Thus,
			\begin{align}\label{traceIneqTwo}
						\left|\int_{\T}\trace(\Phi(\zeta)\Psi(\zeta))\dm(\zeta)\right|&\leq\int_{\T}\left(\sum_{j=0}^{k-1}s_{j}(\Phi(\zeta))\right)
																																					\|\Psi(\zeta)\|_{\M_{n,m}}\dm(\zeta)\nonumber\\
																															&\leq\int_{\T}\left(\sum_{j=0}^{k-1}t_{j}(\Phi)\right)
																																					\|\Psi(\zeta)\|_{\M_{n,m}}\dm(\zeta)\nonumber\\
																															&\leq\left(\sum_{j=0}^{k-1}t_{j}(\Phi)\right)
																																					\|\Psi\|_{L^{1}(\M_{n,m})}\nonumber\\
																															&\leq\sum_{j=0}^{k-1}t_{j}(\Phi),
			\end{align}
			because the singular values of $\Phi$ satisfy $s_{j}(\Phi(\zeta))=t_{j}(\Phi)$ for a.e. $\zeta\in\T$
			since $\Phi$ is very badly approximable.
\end{proof}

Before proceeding, let us observe that equality holds in $(\ref{traceIneq})$ for some simple cases.  Let 
$r$ be a positive integer and $t_{0}, t_{1},\ldots, t_{r-1}$ be positive numbers satisfying
\[	t_{0}\geq t_{1}\geq\ldots\geq t_{r-1}.	\]
Suppose $\Phi$ is an $n\times n$ matrix function of the form
\begin{equation}\label{diagPhi}
			\Phi\Mydef\left(
			\begin{array}{ccccc}
						t_{0}u_{0}	&	\Zero				&	\ldots	&	\Zero						&	\Zero\\
						\Zero				&	t_{1}u_{1}	&	\ldots	&	\Zero						&	\Zero\\
						\vdots			&	\vdots			&	\ddots	&	\vdots					&	\vdots\\
						\Zero				&	\Zero				&	\ldots	&	t_{r-1}u_{r-1}	&	\Zero\\
						\Zero				&	\Zero				&	\ldots	&	\Zero						&	\Phi_{\#}
			\end{array}\right),
\end{equation}
where $\|\Phi_{\#}\|_{L^{\infty}}\leq t_{r-1}$ and $u_{j}$ is a unimodular function of the form 
$u_{j}=\bar{z}\bar{\theta}_{j}\bar{h}_{j}/h_{j}$ with $\theta_{j}$ an inner function and $h_{j}$ 
an outer function in $H^{2}$ for $0\leq j\leq r-1$.
Without loss of generality, we may assume that $\|h_{j}\|_{L^{2}}=1$ for each $j$.  By setting
\begin{equation}\label{diagPsi}
			\Psi\Mydef 
						\left(\begin{array}{ccccc}
												z\theta_{0}h_{0}^{2}	&	\Zero									&	\ldots	&	\Zero											&	\Zero\\
												\Zero									&	z\theta_{1}h_{1}^{2}	&	\ldots	&	\Zero											&	\Zero\\
												\vdots								&	\vdots								&	\ddots	&	\vdots										&	\vdots\\
												\Zero									&	\Zero									&	\ldots	&	z\theta_{r-1}h_{r-1}^{2}	&	\Zero\\
												\Zero									&	\Zero									&	\ldots	&	\Zero											&	\Zero\\
									\end{array}\right),
\end{equation}
it can be seen that $\Psi\in H^{1}_{0}(\M_{n})$, $\rank \Psi(\zeta)=r$ a.e. on $\T$, $\|\Psi\|_{L^{1}(\M_{n})}=1$, and
\[	\int_{\T}{\rm trace}(\Phi(\zeta)\Psi(\zeta))\dm(\zeta)=t_{0}+\ldots+t_{r-1}.	\]
Thus we obtain that
\[	\cS_{r}(\Phi)=t_{0}(\Phi)+\ldots+t_{r-1}(\Phi).	\]

On the other hand, one cannot expect the inequality $(\ref{traceIneq})$ to become an equality in general.
After all, by the Hahn-Banach Theorem,
\begin{equation}\label{distanceFormula}
			\dist_{L^{\infty}(\bS_{1}^{n})}(\Phi, H^{\infty}(\M_{n}))=\cS_{n}(\Phi),
\end{equation}
and there are admissible very badly approximable $2\times 2$ matrix functions $\Phi$ for which the strict inequality
\begin{equation*}
			\dist_{L^{\infty}(\bS_{1}^{2})}(\Phi, H^{\infty}(\M_{2}))<t_{0}(\Phi)+t_{1}(\Phi)
\end{equation*}
holds.  For instance, consider the matrix function
\begin{equation*}
			\Phi	=\left(	\begin{array}{cc}
												\bar{z}	&	\Zero\\
												\Zero		&	\bar{z}									
									\end{array}\right)
						\frac{1}{\sqrt{2}}\left(	\begin{array}{cc}
												1		&	\bar{z}\\
												-z	&	1									
									\end{array}\right)
					=\frac{1}{\sqrt{2}}\left(	\begin{array}{cc}
																					\bar{z}	&	\bar{z}^{2}\\
																					-1			&	\bar{z}									
																		\end{array}\right).
\end{equation*}
Clearly, $\Phi$ has superoptimal singular values $t_{0}(U)=t_{1}(U)=1$.  Let
\begin{equation*}
			F=\frac{1}{\sqrt{2}}\left(\begin{array}{cc}
																			\Zero	&	\Zero\\
																			-1		&	\Zero
																\end{array}\right).
\end{equation*}
It is not difficult to verify that
\[	s_{0}((\Phi-F)(\zeta))=\frac{1}{2}\sqrt{3+\sqrt{5}}\,\mbox{ and }\;s_{1}((\Phi-F)(\zeta))=\frac{1}{2}\sqrt{3-\sqrt{5}}	\]
for all $\zeta\in\T$.  Therefore
\begin{equation}\label{strictIneq}
			\dist_{L^{\infty}(\bS_{1}^{2})}(\Phi, H^{\infty}(\M_{2}))\leq\|\Phi-F\|_{L^{\infty}(\bS_{1}^{2})}<2=t_{0}(\Phi)+t_{1}(\Phi).											
\end{equation}

\subsection{What is done in this paper?}\label{doneInThisPaper}

In virtue of Theorem \ref{SInequality} and the remarks proceeding it, one may ask whether it is possible to 
characterize the matrix functions $\Phi$ for which $(\ref{traceIneq})$ becomes an equality. So let $\Phi$ be an 
admissible $n\times n$ matrix function with a superoptimal approximant $Q$ in $H^{\infty}(\M_{n})$ for which 
equality in Theorem \ref{SInequality} holds with $k=n$.  In this case, it must be that
\[	\dist_{L^{\infty}(\bS_{1}^{n})}(\Phi,H^{\infty}(\M_{n}))=\sum_{j=0}^{n-1}t_{j}(\Phi)
		=\sum_{j=0}^{n-1}s_{j}((\Phi-Q)(\zeta))=\|\Phi-Q\|_{L^{\infty}(\bS_{1}^{n})}	\]
by $(\ref{distanceFormula})$ and thus the superoptimal approximant $Q$ must be a best approximant to $\Phi$ 
under the $L^{\infty}(\bS_{1}^{n})$ norm as well.  Hence, we are led to investigate the following problems:
\begin{enumerate}
			\item	For which matrix functions $\Phi$ does Extremal problem \ref{extremal} have a solution?
			\item	If $Q_{\$}$ is a best approximant to $\Phi$ under the $L^{\infty}(\bS_{1}^{n})$-norm, when does it 
						follow that $Q_{\$}$ is the superoptimal approximant to $\Phi$ in $L^{\infty}(\M_{n})$?
			\item	Can we find necessary and sufficient conditions on $\Phi$ to obtain equality in $(\ref{traceIneq})$ 
						of Theorem \ref{SInequality}?
\end{enumerate}

Before addressing these problems, we recall certain standard principles of functional analysis in Section 
\ref{prelim} that are used throughout the paper.  In particular, we give their explicit formulation for the spaces 
$L^{p}(\bS_{q}^{m,n})$.

In Section \ref{HankelSection}, we introduce the Hankel-type operators $H^{\{k\}}_{\Phi}$ on spaces of matrix 
functions and $k$-extremal functions, and prove that the number $\sigma_{k}(\Phi)$ equals the operator norm of 
$H^{\{k\}}_{\Phi}$.  We also show that Extremal problem \ref{extremal} has a solution if and only if 
the Hankel-type operator $H^{\{k\}}_{\Phi}$ has a maximizing vector, and thus answer question 1 in terms 
Hankel-type operators.

In Section \ref{sumSection}, we establish the main results of this paper concerning best approximation under the 
$L^{\infty}(\bS_{1}^{m,n})$ norm (Theorem \ref{traceIneqisEquality}) and the sum of superoptimal singular values 
(Theorem \ref{classification}).  The latter result characterizes the smallest number $k$ for which
\[	\int_{\T}\trace(\Phi(\zeta)\Psi(\zeta))\dm(\zeta)	\]
equals the sum of all non-zero superoptimal singular values for some function $\Psi\in\cA_{k}^{n,m}$.  These results 
serve as partial solutions to problems 2 and 3.  

Lastly, in Section \ref{unitarySection}, we restrict our attention to unitary-valued very badly approximable matrix 
functions.  For any such matrix function $U$, we provide a representation of any function $\Psi$ for which the formula
\[	\int_{\T}\trace(U(\zeta)\Psi(\zeta))\dm(\zeta)=n	\]
holds.

{\centering \section{Best approximation and dual extremal problems}\label{prelim}}

We now provide explicit formulation of some basic results concerning best
approximation in $H^{q}(\bS_{p}^{m,n})$ for functions in $L^{q}(\bS_{p}^{m,n})$ and the corresponding dual extremal
problem.  We first consider the general setting.

\smallskip

\subsection{Best approximation}\label{bestApproxSection}

\begin{definition}
			Let $X$ be a normed space, $M$ be a closed subspace of $X$, and $x_{0}\in X$.  We say that \emph{$m_{0}$ is a 
			best approximant to $x_{0}$ in $M$} if $m_{0}\in M$ and
			\[	\|x_{0}-m_{0}\|_{X}=\dist(x_{0},M)\Mydef\inf\{ \|x_{0}-m\|_{X}: m\in M \}.	\]
\end{definition}

It is known that \emph{if $X$ is a reflexive Banach space and $M$ is a closed subspace of $X$, then each $x_{0}\in 
X\setminus M$ has a best approximant $m_{0}$ in $M$.}  

Two standard principles from functional analysis are used throughout this note.  Namely, if $X$ is a normed space 
with a linear subspace $M$, then 
\begin{align*}
			\sup_{m\in M, \|m\|\leq 1} |\Lambda_{0}(m)|&=\min\left\{ \|\Lambda_{0}-\Lambda\|: \Lambda\in M^{\bot } \right\}\;\mbox{ and }\\
			\max_{\Lambda\in M^{\bot}, \|\Lambda\|\leq 1} |\Lambda(x_{0})|&=\dist(x_{0},M)\;\mbox{ whenever $M$ is closed.}
\end{align*}

We now discuss these results in the case of the spaces $L^{q}(\bS_{p}^{m,n})$.

\subsection{The spaces $L^{q}(\bS_{p}^{m,n})$}\label{LqRemarks}

Let $1\leq q<\infty$ and $1\leq p\leq \infty$.  \emph{Let $p'$ denote the conjugate exponent to $p$}, 
i.e. $p'=p/(p-1)$.

Let $\bS_{p}^{m,n}$ denote the space of $m\times n$ matrices equipped with the Schatten-von Neumann norm 
$\|\cdot\|_{\bS_{p}^{m,n}}$, i.e. for $A\in\M_{m,n}$
\[	\|A\|_{\bS_{\infty}^{m,n}}\Mydef\|A\|_{\M_{m,n}}\,\mbox{ and }\;
		\|A\|_{\bS_{p}^{m,n}}\Mydef\left(\sum_{j\geq0}s_{j}^{p}(A)\right)^{1/p}\mbox{ for } 1\leq p<\infty.	\]
We also use the notation $\bS_{p}^{n}\Mydef\bS_{p}^{n,n}$.

If $X$ is a normed space of functions on $\T$ with norm $\|\cdot\|_{X}$, then $X(\bS_{p}^{m,n})$ denotes the space
of $m\times n$ matrix functions whose entries belong to $X$.  For $\Phi\in X(\bS_{p}^{m,n})$, we define
\[	\|\Phi\|_{X(\bS_{p}^{m,n})}\Mydef\|\rho\|_{X}\mbox{, where }
		\rho(\zeta)\Mydef\|\Phi(\zeta)\|_{\bS_{p}^{m,n}}\mbox{ for }\zeta\in\T.	\]

It is known that the dual space of $L^{q}(\bS_{p}^{m,n})$ is isometrically isomorphic to $L^{q'}(\bS_{p'}^{n,m})$ via 
the mapping $\Phi\mapsto \Lambda_{\Phi}$, where $\Phi\in L^{q'}(\bS_{p'}^{n,m})$ and
\[	\Lambda_{\Phi}(\Psi)=\int_{\T}\trace(\Phi(\zeta)\Psi(\zeta))\dm(\zeta)\;\mbox{ for }\Psi\in L^{q}(\bS_{p}^{m,n}).	\]
In particular, it follows that the annihilator of $H^{q}(\bS_{p}^{m,n})$ in $L^{q}(\bS_{p}^{m,n})$ is given by 
$H_{0}^{q'}(\bS_{p'}^{n,m})$, and so
\begin{equation*}
			\dist_{L^{q}(\bS_{p}^{m,n})}(\Phi,H^{q}(\bS_{p}^{m,n}))
			=\max_{\|\Psi\|_{H_{0}^{q'}(\bS_{p'}^{n,m})}\leq 1}\left|\int_{\T}\trace(\Phi(\zeta)\Psi(\zeta))\dm(\zeta)\right|,
\end{equation*}
by our remarks in Section \ref{bestApproxSection}. Moreover, if $1<q<\infty$, then $\Phi\in L^{q}(\bS_{p}^{m,n})$ has 
a best approximant $Q$ in $H^{q}(\bS_{p}^{m,n})$ (as $L^{q}(\bS_{p}^{m,n})$ is reflexive); that is,
\[	\|\Phi-Q\|_{L^{q}(\bS_{p}^{m,n})}=\dist_{L^{q}(\bS_{p}^{m,n})}(\Phi,H^{q}(\bS_{p}^{m,n})).	\]

The situation is similar in the case of $L^{\infty}(\bS_{p}^{m,n})$.  Indeed, $L^{\infty}(\bS_{p}^{m,n})$ is a dual
space, and so there is a $Q\in H^{\infty}(\bS_{p}^{m,n})$ such that
\[	\|\Phi-Q\|_{L^{\infty}(\bS_{p}^{m,n})}=\dist_{L^{\infty}(\bS_{p}^{m,n})}(\Phi, H^{\infty}(\bS_{p}^{m,n})).	\]
Again, it also follows from our remarks in Section \ref{bestApproxSection} that
\[	\dist_{L^{\infty}(\bS_{p}^{m,n})}(\Phi, H^{\infty}(\bS_{p}^{m,n}))=
		\sup_{\|\Psi\|_{H_{0}^{1}(\bS_{p'}^{n,m})}\leq 1}\left|\int_{\T}\trace(\Phi(\zeta)\Psi(\zeta))\dm(\zeta)\right|.	\]
However, an extremal function may fail to exist in this case even if $\Phi$ is a scalar-valued function.  An example
can be deduced from Section 1 of Chapter 1 in \cite{Pe1}.

{\centering \section{$\cS_{k}(\Phi)$ as the norm of a Hankel-type operator and $k$-extremal functions}\label{HankelSection}}

We now introduce the Hankel-type operators $H_{\Phi}^{\{k\}}$ which act on spaces of matrix functions.  We prove that 
the number $\cS_{k}(\Phi)$ equals the operator norm of $H_{\Phi}^{\{k\}}$ and characterize when $H_{\Phi}^{\{k\}}$ 
has a maximizing vector.  Recall that for an operator $T:X\rightarrow Y$ between normed spaces $X$ and $Y$, a vector 
$x\in X$ is called a \emph{a maximizing vector of $T$} if $x$ is nonzero and
\[	\|Tx\|_{Y}=\|T\|\cdot\|x\|_{X}.	\]

We begin by establishing the following lemma.

\begin{lemma}\label{factorPsi}
			Let $1\leq k\leq\min\{m,n\}$.  If $\Psi\in H^{1}(\M_{n,m})$ is such that $\rank \Psi(\zeta)=k$ for a.e.
			$\zeta\in\T$, then there are functions $R\in H^{2}(\M_{n,k})$ and $Q\in H^{2}(\M_{k,m})$ such that $R(\zeta)$
			has rank equal to $k$ for almost every $\zeta\in\T$,
			\[	\Psi=RQ\,\mbox{ and }\;
					\|R(\zeta)\|_{\M_{n,k}}^{2}=\|Q(\zeta)\|_{\M_{k,m}}^{2}=\|\Psi(\zeta)\|_{\M_{n,m}}\mbox{ for a.e. }\zeta\in\T.	\]
\end{lemma}
\begin{proof}
			Consider the set
			\[	\mathscr{A}=\clos_{L^{1}(\C^{n})}\{f\in H^{1}(\C^n):\;f(\zeta)\in\Range\Psi(\zeta)\;\mbox{  a.e. on }\T\, \}.	\]
			Since $\mathscr{A}$ is a non-trivial completely non-reducing (closed) invariant subspace of $L^{1}(\C^{n})$, 
			there is an $n\times r$ inner function $\Theta$ such that $\mathscr{A}=\Theta H^{1}(\C^{r})$.  We first show that
			$r=k$. Let $\{e_{j}\}_{j=1}^{r}$ be an orthonormal basis for $\C^{r}$.  Then for almost every $\zeta\in\T$, we
			have that $\{\Theta(\zeta)e_{j}\}_{j=1}^{r}$ is a linearly independent set, since $\Theta$ is inner.  Moreover,
			$\{\Theta(\zeta)e_{j}\}_{j=1}^{r}$ is a basis for $\Range\Theta(\zeta)=\Range\Psi(\zeta)$ for a.e. $\zeta\in\T$.
			Since $\dim\Range \Psi(\zeta)=k$ a.e. on $\T$, it follows that
			$r=\dim\Range\Theta(\zeta)=\dim\Range\Psi(\zeta)=k$.  In particular, we obtain that
			\[	\mathscr{A}=\Theta H^{1}(\C^{k}).	\]
			
			By considering the columns of $\Psi$, it is easy to see that $\Psi=\Theta F$ for some $k\times m$ matrix function 
			$F\in H^{1}(\mathbb{M}_{k,m})$ as these columns belong to $\mathscr{A}$.  Let $h$ be an outer function in $H^{2}$ 
			such that $|h(\zeta)|=\|\Psi(\zeta)\|_{\mathbb{M}_{n,m}}^{1/2}$ for a.e. $\zeta\in\T$.  The conclusion of
			the lemma now follows by considering the functions
			\[	R=h\Theta\,\mbox{ and }\;Q=h^{-1}F.\qedhere	\]
\end{proof}

\begin{definition}
			Let $\Phi\in L^{\infty}(\M_{m,n})$, $1\leq k \leq\min\{m,n\}$, and $\rho:L^{2}(\bS_{1}^{m,k})\rightarrow 
			L^{2}(\bS_{1}^{m,k})/H^{2}(\bS_{1}^{m,k})$ denote the natural quotient map. We define the \emph{Hankel-type operator} 
			$H_{\Phi}^{\{k\}}: H^{2}(\M_{n,k})\rightarrow L^{2}(\bS_{1}^{m,k})/H^{2}(\bS_{1}^{m,k})$ by setting
			\[	H_{\Phi}^{\{k\}}F\Mydef\rho(\Phi F)\;\mbox{ for }F\in H^{2}(\M_{n,k}).	\]
\end{definition}

The norm in the quotient space $L^{2}(\bS_{1}^{m,k})/H^{2}(\bS_{1}^{m,k})$ is the natural one; that is, the norm 
of a coset equals the infimum of the $L^{2}(\bS_{1}^{m,k})$-norms of its elements.

\begin{thm}\label{HankelNorm}
			Let $1\leq k\leq\min\{m,n\}$. If $\Phi\in L^{\infty}(\M_{m,n})$, then
			\[	\cS_{k}(\Phi)=\left\|H_{\Phi}^{\{k\}}\right\|_{H^{2}(\M_{n,k})\rightarrow L^{2}(\bS_{1}^{m,k})/H^{2}(\bS_{1}^{m,k})}.	\]
\end{thm}
\begin{proof}
			Consider the collection
			\[	\cB_{k}^{n,m}=\{	RQ:\,\|R\|_{H^{2}(\M_{n,k})}\leq 1, \|Q\|_{H^{2}_{0}(\M_{k,m})}\leq1 	\}.	\]
			We claim that $\cB_{k}^{n,m}=\cA_{k}^{n,m}$.  Indeed if $\Psi\in \cA_{k}$ satisfies $\rank\Psi(\zeta)=j$ for 
			$\zeta\in\T$, where $1\leq j\leq k$, then by Lemma \ref{factorPsi} there are functions $R\in H^{2}(\M_{n,j})$ 
			and $Q\in H^{2}_{0}(\M_{j,m})$ such that $R(\zeta)$ has rank equal to $j$ for almost every $\zeta\in\T$,
			\[	\Psi=RQ\,\mbox{ and }\;
					\|R(\zeta)\|_{\M_{n,j}}^{2}=\|Q(\zeta)\|_{\M_{j,m}}^{2}=\|\Psi(\zeta)\|_{\M_{n,m}}\mbox{ for a.e. }\zeta\in\T.	\]
			We may now add zeros, if necessary, to obtain $n\times k$ and $k\times m$ matrix functions
			\begin{equation*}
						R_{\#}=(\,R\quad\Zero\,)\,\mbox{ and }\;
						Q_{\#}=\left(	\begin{array}{c}
																Q\\
																\Zero
													\end{array}\right),
			\end{equation*}
			respectively, from which it follows that $\Psi=R_{\#}Q_{\#}\in\cB_{k}^{n,m}$.  Therefore $\cA_{k}^{n,m}
			\subset \cB_{k}^{n,m}$. The reverse inclusion is trivial and so these sets are equal.
			  
			Hence
			\begin{align*}
						\cS_{k}(\Phi)	&=\sup_{\|R\|_{H^{2}(\mathbb{M}_{n,k})}\leq 1}\;\sup_{\|Q\|_{H^{2}_{0}(\mathbb{M}_{k,m})}\leq 1}
														\left|\int_{\T}{\rm trace}(\Phi(\zeta)R(\zeta)Q(\zeta))\dm(\zeta)\right|\\
													&=\sup_{\|R\|_{H^{2}(\mathbb{M}_{n,k})}\leq 1}\;\dist_{L^{2}(\bS_{1}^{m,k})}
															(\Phi R, H^{2}(\mathbb{M}_{m,k}))\\
													&=\|H^{\{k\}}_{\Phi}\|_{H^{2}(\mathbb{M}_{n,k})\rightarrow L^{2}(\bS_{1}^{m,k})/H^{2}(\bS_{1}^{m,k})}.\qedhere
			\end{align*}
\end{proof}

\begin{definition}
			Let $\Phi\in L^{\infty}(\M_{m,n})$ and $1\leq k \leq\min\{m,n\}$.  We say that $\Psi$ is \emph{a 
			$k$-extremal function for $\Phi$} if $\Psi\in\cA_{k}^{n,m}$ and
			\[	\cS_{k}(\Phi)=\int_{\T}\trace(\Phi(\zeta)\Psi(\zeta))\dm(\zeta).	\]
\end{definition}

Thus a matrix function $\Phi$ has a $k$-extremal function if and only if Extremal problem \ref{extremal} 
has a solution.  

We can now describe matrix functions that have a $k$-extremal function in terms of Hankel-type operators.

\begin{thm}\label{kExtMax}
			Let $\Phi\in L^{\infty}(\M_{m,n})$.  The matrix function $\Phi$ has a $k$-extremal function if and only if
			the Hankel-type operator $H^{\{k\}}_{\Phi}: H^{2}(\M_{n,k})\rightarrow L^{2}(\bS_{1}^{m,k})/H^{2}(\bS_{1}^{m,k})$ 
			has a maximizing vector.
\end{thm}
\begin{proof}
			To simplify notation, let 
			\[	\left\|H_{\Phi}^{\{k\}}\right\|\Mydef
					\left\|H_{\Phi}^{\{k\}}\right\|_{H^{2}(\M_{n,k})\rightarrow L^{2}(\bS_{1}^{m,k})/H^{2}(\bS_{1}^{m,k})}.	\] 
			
			Suppose $\Psi$ is a $k$-extremal function for $\Phi$.  Let $j\in\N$ be such that $j\leq k$ and
			\[	\rank\Psi(\zeta)=j\mbox{ for a.e. } \zeta\in\T.	\]			
			By Lemma \ref{factorPsi}, there is an $R\in H^{2}(\M_{n,j})$ and a $Q\in H^{2}_{0}(\M_{j,m})$ such that 
			\[	\Psi=RQ\,\mbox{ and }\;
					\|R(\zeta)\|_{\M_{n,j}}^{2}=\|Q(\zeta)\|_{\M_{j,m}}^{2}=\|\Psi(\zeta)\|_{\M_{n,m}}\mbox{ for a.e. }\zeta\in\T.	\]
			As before, adding zeros if necessary, we obtain $n\times k$ and $k\times m$ matrix functions
			\begin{equation*}
						R_{\#}=(\,R\quad\Zero\,)\,\mbox{ and }\;
						Q_{\#}=\left(	\begin{array}{c}
																Q\\
																\Zero
													\end{array}\right),
			\end{equation*}
			respectively, so that $\Psi=R_{\#}Q_{\#}$ and
			\[	\|Q_{\#}(\zeta)\|_{\M_{k,m}}^{2}=\|Q(\zeta)\|_{\M_{j,m}}^{2}=\|\Psi(\zeta)\|_{\M_{n,m}}\mbox{ for a.e. }\zeta\in\T.	\]
					
			Let us show that $R_{\#}$ is a maximizing vector for $H_{\Phi}^{\{k\}}$.  Since $Q_{\#}$ belongs to $H^{2}_{0}(\M_{k,m})$,
			we have that for any $F\in H^{2}(\bS_{1}^{m,k})$
			\begin{align*}
						\cS_{k}(\Phi)&=\int_{\T}\trace(\Phi(\zeta)\Psi(\zeta))\dm(\zeta)
										=\int_{\T}\trace(\Phi(\zeta)R_{\#}(\zeta)Q_{\#}(\zeta))\dm(\zeta)\\
								&=\int_{\T}\trace((\Phi R_{\#}-F)(\zeta)Q_{\#}(\zeta))\dm(\zeta),
			\end{align*}
			and so
			\begin{align*}
						\cS_{k}(\Phi)&=\left|\int_{\T}\trace((\Phi R_{\#}-F)(\zeta)Q_{\#}(\zeta))\dm(\zeta)\right|\\
												&\leq\int_{\T}\left|\trace((\Phi R_{\#}-F)(\zeta)Q_{\#}(\zeta))\right|\dm(\zeta)\\
												&\leq\int_{\T}\|(\Phi R_{\#}-F)(\zeta)Q_{\#}(\zeta)\|_{\bS_{1}^{m}}\dm(\zeta)\\
												&\leq\int_{\T}\|(\Phi R_{\#}-F)(\zeta)\|_{\bS_{1}^{m,k}}\|Q_{\#}(\zeta)\|_{\M_{k,m}}\dm(\zeta)\\
												&\leq\|\Phi R_{\#}-F\|_{L^{2}(\bS_{1}^{m,k})}\|Q_{\#}\|_{L^{2}(\M_{k,m})}\\
												&=\|\Phi R_{\#}-F\|_{L^{2}(\bS_{1}^{m,k})}\|\Psi\|_{L^{1}(\M_{n,m})}\\
												&\leq\|\Phi R_{\#}-F\|_{L^{2}(\bS_{1}^{m,k})}.
			\end{align*}	
			By Theorem \ref{HankelNorm}, we obtain that
			\[	\cS_{k}(\Phi)\leq\left\|H_{\Phi}^{\{k\}}R_{\#}\right\|_{L^{2}(\bS_{1}^{m,k})/H^{2}(\bS_{1}^{m,k})}
					\leq\left\|H_{\Phi}^{\{k\}}\right\|=\cS_{k}(\Phi),	\]
			and therefore
			\[	\left\|H_{\Phi}^{\{k\}}\right\|=\left\|H_{\Phi}^{\{k\}}R_{\#}\right\|_{L^{2}(\bS_{1}^{m,k})/H^{2}(\bS_{1}^{m,k})}.	\]
			Thus, $R_{\#}$ is a maximizing vector of $H_{\Phi}$.

			Conversely, suppose the Hankel-type operator $H_{\Phi}^{\{k\}}$ has a maximizing vector $R\in H^{2}(\M_{n,k})$.
			Without loss of generality, we may assume that $\|R\|_{L^{2}(\M_{n,k})}=1$.	Then
			\[	\dist_{L^{2}(\bS_{1}^{m,k})}(\Phi R, H^{2}(\bS_{1}^{m,k}))=\left\|H_{\Phi}^{\{k\}}\right\|.	\]		
						
			By the remarks in Section \ref{LqRemarks}, there is a function $G\in H^{2}_{0}(\M_{k,m})$ such that
			$\|G\|_{L^{2}(\M_{k,m})}\leq 1$ and 
			\[	\int_{\T}\trace((\Phi R)(\zeta)G(\zeta))\dm(\zeta)=\dist_{L^{2}(\bS_{1}^{m,k})}(\Phi R, H^{2}(\bS_{1}^{m,k})).	\]
			On the other hand, since $R$ is a maximizing vector of $H_{\Phi}^{\{k\}}$, it follows from Theorem \ref{HankelNorm} that
			\[	\int_{\T}\trace(\Phi(\zeta) (RG)(\zeta))\dm(\zeta)=\left\|H_{\Phi}^{\{k\}}\right\|=\cS_{k}(\Phi).	\]
			Hence $\Psi\Mydef RG$ is a $k$-extremal function for $\Phi$.
\end{proof}

Before stating the next result, let us recall that the Hankel operator $H_{\Phi}:H^{2}(\C^{n})\rightarrow H^{2}_{-}(\C^{m})$ 
is defined by $H_{\Phi}f=\pP_{-}\Phi f$ for $f\in H^{2}(\C^{n})$.  The following is an immediate consequence of the previous 
theorem when $k=1$.  

\begin{cor}\label{vectorCase}
			Let $\Phi\in L^{\infty}(\M_{m,n})$.  The Hankel operator $H_{\Phi}$ has a maximizing vector if and only 
			if $\Phi$ has a $1$-extremal function.
\end{cor}
\begin{proof}
			By Theorem \ref{kExtMax}, $\Phi$ has a 1-extremal function if and only if the Hankel-type operator
			$H^{\{1\}}_{\Phi}: H^{2}(\C^{n})\rightarrow L^{2}(\C^{m})/H^{2}(\C^{m})$ has a maximizing vector.  The conclusion
			now follows by considering the ``natural'' isometric isomorphism between the spaces $H_{-}^{2}(\C^{m})=L^{2}(\C^{m})
			\ominus H^{2}(\C^{m})$ and $L^{2}(\C^{m})/H^{2}(\C^{m})$.
\end{proof}

\begin{remark}\label{vectorCaseRem}
			It is worth mentioning that if a matrix function $\Phi$ is such that the Hankel operator $H_{\Phi}$ has a maximizing vector 
			(e.g. $\Phi\in (H^{\infty}+C)(\M_{n})$), then any $1$-extremal function $\Psi$ of $\Phi$ satisfies
			\[	\int_{\T}\trace(\Phi(\zeta)\Psi(\zeta))dm(\zeta)=\left\|H_{\Phi}\right\|=t_{0}(\Phi).	\]
			This is a consequence of Corollary \ref{vectorCase} and Theorem \ref{HankelNorm}.
\end{remark}	

\begin{remark}
			There are other characterizations of the class of bounded matrix functions $\Phi$ such that the Hankel operator $H_{\Phi}$ 
			has a maximizing vector.  These involve ``dual'' extremal functions and ``thematic'' factorizations.  We refer 
			the interested reader to \cite{Pe2} for details.
\end{remark}		

\begin{cor}\label{maxVectors}
			Let $1\leq k\leq \ell\leq n$ and $\Phi\in L^{\infty}(\M_{n})$.  Suppose that $\cS_{k}(\Phi)=\cS_{\ell}(\Phi)$.
			If $H_{\Phi}^{\{k\}}$ has a maximizing vector, then $H^{\{\ell\}}_{\Phi}$ also has a maximizing vector.
\end{cor}
\begin{proof}
			This is an immediate consequence of Theorem \ref{kExtMax}.
\end{proof}

{\centering \section{How about the sum of superoptimal singular values?}\label{sumSection}}

In this section, we prove in Theorem \ref{traceIneqisEquality} that equality is obtained in $(\ref{traceIneq})$ under 
some natural conditions.  

For the rest of this note, we assume that $m=n$.  

Consider the non-decreasing sequence $\cS_{1}(\Phi),\ldots,\cS_{n}(\Phi)$.  Recall that
\[	\cS_{n}(\Phi)=\dist_{L^{\infty}(\bS_{1}^{n})}(\Phi, H^{\infty}(\M_{n}))	\]	
and the distance on the right-hand side is in fact always attained, i.e. a best approximant $Q$ to $\Phi$ under 
the $L^{\infty}(\bS_{1}^{n})$ norm always exists as explained in Section \ref{LqRemarks}.\smallskip

\begin{thm}\label{bestApproxProp}
			Let $\Phi\in L^{\infty}(\M_{n})$ and $1\leq k\leq n$.  Suppose $Q$ is a best approximant to $\Phi$ in
			$H^{\infty}(\M_{n})$ under the $L^{\infty}(\bS_{1}^{n})$-norm.  If the Hankel-type operator $H_{\Phi}^{\{k\}}$
			has a maximizing vector $\cF$ in $H^{2}(\M_{n,k})$ and $\cS_{k}(\Phi)=\cS_{n}(\Phi)$, then
			\begin{enumerate}
						\item	$Q\cF$ is a best approximant to $\Phi \cF$ in $H^{2}$ under the $L^{2}(\bS_{1}^{n,k})$-norm,
						\item	for each $j\geq0$,
						\[	s_{j}((\Phi-Q)(\zeta)\cF(\zeta))=s_{j}((\Phi-Q)(\zeta))\|\cF(\zeta)\|_{\M_{n,k}}\mbox{ for a.e. }\zeta\in\T,	\]
						\item	$\displaystyle{\sum_{j=0}^{k-1}s_{j}((\Phi-Q)(\zeta)=\cS_{k}(\Phi)}$ holds for a.e. $\zeta\in\T$, and
						\item	$s_{j}((\Phi-Q)(\zeta))=0$ holds for a.e. $\zeta\in\T$ whenever $j\geq k$.
			\end{enumerate}
\end{thm}
\begin{proof}
			By our assumptions, 
			\begin{align*}
						\|H^{\{k\}}_{\Phi}\|^{2}\|\cF\|_{L^{2}(\M_{n,k})}^{2}
															&=\|H^{\{k\}}_{\Phi}\cF\|^{2}_{L^{2}(\bS_{1}^{n,k})/H^{2}(\bS_{1}^{n,k})}
																=\|\rho(\Phi\cF)\|^{2}\\
														&=\|\rho((\Phi-Q)\cF)\|^{2}\\
														&\leq\|(\Phi-Q) \cF\|_{L^{2}(\bS_{1}^{n,k})}^{2}
																=\int_{\T}\|(\Phi-Q)(\zeta)\cF(\zeta)\|_{\bS_{1}^{n,k}}^{2}\dm(\zeta)\\
														&\leq\int_{\T}\|(\Phi-Q)(\zeta)\|_{\bS_{1}^{n}}^{2}\|\cF(\zeta)\|_{\M_{n,k}}^{2}\dm(\zeta)\\
														&\leq\|\Phi-Q\|_{L^{\infty}(\bS_{1}^{n})}^{2}\|\cF\|_{L^{2}(\M_{n,k})}^{2}
																=\cS_{k}(\Phi)^{2}\|\cF\|_{L^{2}(\M_{n,k})}^{2}.
			\end{align*}
			It follows from Theorem \ref{HankelNorm} that all inequalities are equalities.  In particular, we obtain that $Q\cF$ 
			is a best approximant to $\Phi Q$ under the $L^{2}(\bS_{1}^{n,k})$-norm since the first inequality is actually an 
			equality.  For almost every $\zeta\in\T$,
			\begin{align}
						\|(\Phi-Q)(\zeta)\cF(\zeta)\|_{\bS_{1}^{n}}&=\|(\Phi-Q)(\zeta)\|_{\bS_{1}^{n}}\|\cF(\zeta)\|_{\M_{n,k}}\mbox{ and }\label{first}\\
						\|(\Phi-Q)(\zeta)\|_{\bS_{1}^{n}}&=\|\Phi-Q\|_{L^{\infty}(\bS_{1}^{n})}=\cS_{k}(\Phi),\nonumber
			\end{align}
			because the second and third inequalities are equalities as well.  It follows from $(\ref{first})$ that for each $j\geq0$,
			\[	s_{j}((\Phi-Q)(\zeta)\cF(\zeta))=s_{j}((\Phi-Q)(\zeta))\|\cF(\zeta)\|_{\M_{n,k}}\mbox{ for a.e. }\zeta\in\T.	\]
			
			We claim that if $j\geq k$, then $s_{j}((\Phi-Q)(\zeta))=0$ for a.e. $\zeta\in\T$.  By Theorem \ref{kExtMax},
			we can choose a $k$-extremal function, say $\Psi$, for $\Phi$.  Since $\Psi$ belongs to $H_{0}^{1}(\M_{n})$,
			\begin{align*}
						\cS_{k}(\Phi)&=\int_{\T}\trace(\Phi(\zeta)\Psi(\zeta))\dm(\zeta)=\int_{\T}\trace((\Phi-Q)(\zeta)\Psi(\zeta))\dm(\zeta)\\
												&\leq\int_{\T}\|(\Phi-Q)(\zeta)\Psi(\zeta)\|_{\bS_{1}^{n}}\dm(\zeta)
													\leq\int_{\T}\|(\Phi-Q)(\zeta)\|_{\bS_{1}^{n}}\|\Psi(\zeta)\|_{\M_{n}}\dm(\zeta)\\
												&\leq\|\Phi-Q\|_{L^{\infty}(\bS_{1}^{n})}\|\Psi\|_{L^{1}(\M_{n})}\leq\|\Phi-Q\|_{L^{\infty}(\bS_{1}^{n})}=\cS_{k}(\Phi),
			\end{align*}
			and so all inequalities are equalities.  It follows that 
			\begin{equation}\label{eqBNP}
						|\trace((\Phi-Q)(\zeta)\Psi(\zeta))|=\|(\Phi-Q)(\zeta)\|_{\bS_{1}^{n}}\|\Psi(\zeta)\|_{\M_{n}}\mbox{ for a.e. }\zeta\in\T.
			\end{equation}
			
			In order to complete the proof, we need the following lemma.
			
			\begin{lemma}\label{BNPLemma}
						Let $A\in \M_{n}$ and $B\in\M_{n}$.  Suppose that $A$ and $B$ satisfy
						\[	|\trace(AB)|=\|A\|_{\M_{n}}\|B\|_{\bS_{1}^{n}}.	\]
						If $\rank A\leq k$, then $\rank B\leq k$ as well.
			\end{lemma}
			
			We first finish the proof of Theorem \ref{bestApproxProp} before proving Lemma \ref{BNPLemma}.
			
			It follows from $(\ref{eqBNP})$ and Lemma \ref{BNPLemma} that 
			\[	\rank((\Phi-Q)(\zeta))\leq k\,\mbox{ for a.e. }\zeta\in\T.	\]
			In particular, if $j\geq k$, then
			\[	s_{j}((\Phi-Q)(\zeta))=0\,\mbox{ for a.e. }\zeta\in\T,	\]
			and so
			\[	\sum_{j=0}^{k-1}s_{j}((\Phi-Q)(\zeta))=\|(\Phi-Q)(\zeta)\|_{\bS_{1}^{n}}=\cS_{k}(\Phi)\,\mbox{ for a.e. }\zeta\in\T.	\]
			This completes the proof.
\end{proof}

\begin{remark}
			Lemma \ref{BNPLemma} is a slight modification of Lemma 4.6 in \cite{BNP}.  Although the proof of Lemma \ref{BNPLemma} 
			given below is almost the same as that given in \cite{BNP} for Lemma 4.6, we include it for the convenience of the reader.  
\end{remark}

\begin{proof}[Proof of Lemma \ref{BNPLemma}]
			Let $B$ have polar decomposition $B=UP$ and set $C=AU$, where $P=(B^{*}B)^{1/2}$.  Let 
			$e_{1},\ldots,e_{n}$ be an orthonormal basis of eigenvectors for $P$ and $Pe_{j}=\lambda_{j}e_{j}$.
			It is easy to see that the following inequalities hold:
			\begin{align*}
						|\trace(AB)|&=|\trace(CP)|=\left|\sum_{j=1}^{n}(Pe_{j},C^{*}e_{j})\right|
													=\left|\sum_{j=1}^{n}\lambda_{j}(e_{j},C^{*}e_{j})\right|\\
												&=\left|\sum_{j=1}^{n}\lambda_{j}(Ce_{j},e_{j})\right|
													\leq\sum_{j=1}^{n}\lambda_{j}\left|(Ce_{j},e_{j})\right|
													\leq\sum_{j=1}^{n}\lambda_{j}\|Ce_{j}\|\\
												&\leq\|C\|_{\M_{n}}\sum_{j=1}^{n}\lambda_{j}.
			\end{align*}
			On the other hand, 
			\[	\|A\|_{\M_{n}}\|B\|_{\bS_{1}^{n}}=\|C\|_{\M_{m}}\|P\|_{\bS_{1}^{n}}=\|C\|_{\M_{n}}\sum_{j=1}^{n}\lambda_{j}	\]
			and so, by the assumption $|\trace(AB)|=\|A\|_{\M_{n}}\|B\|_{\bS_{1}^{n}}$, it follows that
			\[	\sum_{j=1}^{n}\lambda_{j}\|Ce_{j}\|=\|C\|_{\M_{n}}\sum_{j=1}^{n}\lambda_{j}.	\]
			Therefore $\lambda_{j}\|Ce_{j}\|=\|C\|_{\M_{n}}\lambda_{j}$ for each $j$.  However, if $\rank A\leq k$, then 
			$\rank C\leq k$. Thus there are at most $k$ vectors $e_{j}$ such that $\|Ce_{j}\|=\|C\|_{\M_{n}}$.  In
			particular, there are at least $n-k$ vectors $e_{j}$ such that $\|Ce_{j}\|<\|C\|_{\M_{n}}$.  Thus,
			$\lambda_{j}=0$ for those $n-k$ vectors $e_{j}$, $\rank P\leq k$, and so $\rank B\leq k$.
\end{proof}

\begin{remark}
			Note that the distance function $d_{\Phi}$ defined on $\T$ by 
			\[	d_{\Phi}(\zeta)\Mydef\|(\Phi-Q)(\zeta)\|_{\bS_{1}^{n}}	\] 
			equals $\sigma_{k}(\Phi)$ for almost every $\zeta\in\T$ and is therefore independent of the choice of the best 
			approximant $Q$.  This is an immediate consequence of Theorem \ref{bestApproxProp}.  A similar phenomenon occurs 
			in the case of matrix functions $\Phi\in L^{p}(\M_{n})$ for $2<p<\infty$.  We refer the reader to \cite{BNP} for details.
\end{remark}

\begin{cor}\label{sumInequality}
			Let $\Phi\in L^{\infty}(\M_{n})$ be an admissible matrix function and $1\leq k\leq n$.  If the Hankel-type 
			operator $H_{\Phi}^{\{k\}}$ has a maximizing vector and $\sigma_{k}(\Phi)=\sigma_{n}(\Phi)$, then
			\[	\sum_{j=0}^{k-1}s_{j}((\Phi-Q)(\zeta))\leq\sum_{j=0}^{k-1}t_{j}(\Phi)	\]
			for any best approximation $Q$ of $\Phi$ in $H^{\infty}(\M_{n})$ under the $L^{\infty}(\bS_{1}^{n})$-norm.
\end{cor}
\begin{proof}
			This is an immediate consequence of Theorems \ref{SInequality} and \ref{bestApproxProp}.
\end{proof}

\begin{definition}
			A matrix function $\Phi\in L^{\infty}(\M_{n})$ is said to have \emph{order $\ell$} if $\ell$ is the smallest 
			number such that $H_{\Phi}^{\{\ell\}}$ has a maximizing vector and
			\[	\cS_{\ell}(\Phi)=\dist_{L^{\infty}(\bS_{1}^{n})}(\Phi, H^{\infty}(\M_{n})).	\]
			If no such number $\ell$ exists, we say that $\Phi$ is \emph{inaccessible}.
\end{definition}

The interested reader should compare this definition of ``order'' with the one made in \cite{BNP} for matrix functions 
in $L^{p}(\M_{n})$ for $2<p<\infty$.  Also, due to Corollary \ref{maxVectors}, it is clear that if $\Phi\in 
L^{\infty}(\M_{n})$ has order $\ell$, then the Hankel-type operator $H_{\Phi}^{\{k\}}$ has a maximizing vector and 
\[	\cS_{k}(\Phi)=\dist_{L^{\infty}(\bS_{1}^{n})}(\Phi, H^{\infty}(\M_{n}))	\] 
holds for each $k\geq \ell$.

\begin{thm}\label{traceIneqisEquality}
			Let $\Phi\in L^{\infty}(\M_{n})$ be an admissible matrix function of order $k$.  The following
			statements are equivalent.
			\begin{enumerate}
						\item	$Q\in H^{\infty}$ is a best approximant to $\Phi$ under the $L^{\infty}(\bS_{1}^{n})$-norm 
									and the functions
									\[	\zeta\mapsto s_{j}((\Phi-Q)(\zeta)), \;0\leq j\leq k-1, \]
									are constant almost everywhere on $\T$.
						\item	$Q$ is the superoptimal approximant to $\Phi$, $t_{j}(\Phi)=0$ for $j\geq k$, and 
									\[	\cS_{k}(\Phi)=t_{0}(\Phi)+\ldots+t_{k-1}(\Phi).	\]
			\end{enumerate}
\end{thm}
\begin{proof}		
			We first prove that \emph{1} implies \emph{2}.  By Corollary \ref{sumInequality}, we have that, for almost every 
			$\zeta\in\T$,
			\[	\sum_{j=0}^{k-1}s_{j}((\Phi-Q)(\zeta))\leq\sum_{j=0}^{k-1}t_{j}(\Phi)
					\leq\sum_{j=0}^{k-1}\ess\sup_{\zeta\in\T}s_{j}((\Phi-Q)(\zeta))=\sum_{j=0}^{k-1}s_{j}((\Phi-Q)(\zeta)).	\]
			This implies that
			\[	t_{j}(\Phi)=\ess\sup_{\zeta\in\T}s_{j}((\Phi-Q)(\zeta))=s_{j}((\Phi-Q)(\zeta))\;\mbox{ for }0\leq j\leq k-1,	\]
			$Q\in\Omega_{k-1}(\Phi)$, and
			\[	\sum_{j=0}^{k-1}t_{j}(\Phi)=\sum_{j=0}^{k-1}s_{j}((\Phi-Q)(\zeta))=\cS_{k}(\Phi).	\]
			Moreover, Theorem \ref{bestApproxProp} gives that $s_{j}((\Phi-Q)(\zeta))=0$ a.e. on $\T$ for $j\geq k$, and so
			$t_{j}(\Phi)=0$ for $j\geq k$, as $Q\in\Omega_{k-1}(\Phi)$.  Hence, $Q$ is the superoptimal approximant to $\Phi$.
			
			Let us show that \emph{2} implies \emph{1}.  Clearly, it suffices to show that if \emph{2} holds, then $Q$ is a
			best approximant to $\Phi$ under the $L^{\infty}(\bS_{1}^{n})$-norm.  Suppose \emph{2} holds.  In this case, we
			must have that
			\[	\cS_{k}(\Phi)=\sum_{j=0}^{k-1}t_{j}(\Phi)=\sum_{j=0}^{k-1}s_{j}((\Phi-Q)(\zeta))
											=\|\Phi-Q\|_{L^{\infty}(\bS_{1}^{n})}.	\]
			Since $\Phi$ has order $k$, it follows that
			\[	\cS_{n}(\Phi)=\|\Phi-Q\|_{L^{\infty}(\bS_{1}^{n})}	\]
			and so the proof is complete.
\end{proof}

For the rest of this section, we restrict ourselves to admissible matrix functions $\Phi$ which are also very
badly approximable.  Recall that, in this case, the function $\zeta\mapsto s_{j}(\Phi(\zeta))$ equals $t_{j}(\Phi)$ a.e.
on $\T$ for $0\leq j\leq n-1$, as mentioned in Section \ref{superApproxSection}.  The next result follows at once from
Theorem \ref{traceIneqisEquality}.

\begin{cor}\label{mainCor}
			Let $\Phi$ be an admissible very badly approximable $n\times n$ matrix function of order $k$.  The zero matrix 
			function is a best approximant to $\Phi$ under the $L^{\infty}(\bS_{1}^{n})$-norm if and only if $t_{j}(\Phi)=0$ 
			for $j\geq k$ and
			\[	\cS_{k}(\Phi)=t_{0}(\Phi)+\ldots+t_{k-1}(\Phi).	\]
\end{cor}

It is natural to question at this point whether or not the collection of admissible very badly approximable matrix functions
of order $k$ is non-empty.  It turns out that one can easily construct examples of admissible very badly approximable
matrix functions of order $k$ (see Examples \ref{simpleEx1} and \ref{simpleEx2}).  Theorem \ref{suffCondition} below
gives a simple sufficient condition for determining when a very badly approximable matrix function has order $k$.  We
first need the following lemma.

\begin{lemma}\label{suffCondLemma}
			Let $\Phi\in L^{\infty}(\M_{n})$.  Suppose there is $\Psi\in\cA_{k}^{n}$ such that
			\[	\int_{\T}\trace(\Phi(\zeta)\Psi(\zeta))\dm(\zeta)=\|\Phi\|_{L^{\infty}(\bS_{1}^{n})}.	\]
			Then $\Psi$ is a $k$-extremal function for $\Phi$, $\cS_{k}(\Phi)=\cS_{n}(\Phi)$, and the zero matrix function 
			is a best approximant to $\Phi$ under the $L^{\infty}(\bS_{1}^{n})$-norm.
\end{lemma}
\begin{proof}
		By the assumptions on $\Psi$, we have
		\[	\|\Phi\|_{L^{\infty}(\bS_{1}^{n})}=\int_{\T}\trace(\Phi(\zeta)\Psi(\zeta))\dm(\zeta)\leq\cS_{k}(\Phi).	\]
		On the other hand, 
		\[	\cS_{k}(\Phi)\leq\dist_{L^{\infty}(\bS_{1}^{n})}(\Phi,H^{\infty})\leq\|\Phi\|_{L^{\infty}(\bS_{1}^{n})}	\]
		always holds. Since all the previously mentioned inequalities are equalities, the conclusion follows.  
\end{proof}

\begin{thm}\label{suffCondition}
			Let $\Phi\in L^{\infty}(\M_{n})$ be an admissible very badly approximable matrix function.  Suppose 
			there is $\Psi\in\cA_{k}^{n}$ such that
			\[	\int_{\T}\trace(\Phi(\zeta)\Psi(\zeta))\dm(\zeta)=t_{0}(\Phi)+\ldots+t_{n}(\Phi).	\]
			If $t_{k-1}(\Phi)>0$, then $\Phi$ has order $k$ and the zero matrix function is a best approximant to 
			$\Phi$ under the $L^{\infty}(\bS_{1}^{n})$-norm.
\end{thm}
\begin{proof}
			By the remarks preceding Corollary \ref{mainCor}, it is easy to see that
			\[	\|\Phi\|_{L^{\infty}(\bS_{1}^{n})}=t_{0}(\Phi)+\ldots+t_{n}(\Phi).	\]
			It follows from Lemma \ref{suffCondLemma} that $\Psi$ is a $k$-extremal function for $\Phi$, 
			$\cS_{k}(\Phi)=\cS_{n}(\Phi)$, and the zero matrix function is a best approximant to $\Phi$ 
			under the $L^{\infty}(\bS_{1}^{n})$-norm.  Thus $\|\Phi\|_{L^{\infty}(\bS_{1}^{n})}=\sigma_{k}(\Phi)$.
			Moreover, by Theorem \ref{SInequality},
			\[	\cS_{k-1}(\Phi)\leq t_{0}(\Phi)+\ldots+t_{k-2}(\Phi)<t_{0}(\Phi)+\ldots+t_{k-1}(\Phi)
					\leq\|\Phi\|_{L^{\infty}(\bS_{1}^{n})}.	\]
			Therefore $\cS_{k-1}(\Phi)<\cS_{k}(\Phi)$.
\end{proof}

\begin{remark}
			Notice that under the hypotheses of Theorem \ref{suffCondition}, one also obtains that $t_{k-1}(\Phi)$
			is the smallest non-zero superoptimal singular value of $\Phi$.  This is an immediate consequence of
			Corollary \ref{mainCor}.
\end{remark}

We now formulate the corresponding result for admissible very badly approximable unitary-valued matrix functions.
These functions are considered in greater detail in Section \ref{unitarySection}.

\begin{cor}\label{unitaryCor}
			Let $U\in L^{\infty}(\M_{n})$ be an admissible very badly approximable unitary-valued matrix function.  
			If there is $\Psi\in\cA_{n}^{n}$ such that
			\[	\int_{\T}\trace(U(\zeta)\Psi(\zeta))\dm(\zeta)=n,	\]
			then $U$ has order $n$ and the zero matrix function is a best approximant to $U$ under the 
			$L^{\infty}(\bS_{1}^{n})$-norm.
\end{cor}
\begin{proof}
			This is a trivial consequence of Theorem \ref{suffCondition} and the fact that
			\[	t_{j}(U)=1 \mbox{ for }0\leq j\leq n-1.\qedhere	\]
\end{proof}

We are now ready to state the main result of this section.

\begin{thm}\label{classification}
			Let $\Phi$ be an admissible very badly approximable $n\times n$ matrix function.  The following statements 
			are equivalent:
			\begin{enumerate}
						\item	$k$ is the smallest number for which there exists $\Psi\in\cA_{k}^{n}$ such that
									\[	\int_{\T}\trace(\Phi(\zeta)\Psi(\zeta))\dm(\zeta)=t_{0}(\Phi)+\ldots+t_{n-1}(\Phi);	\]
						\item	$\Phi$ has order $k$, $t_{j}(\Phi)=0$ for $j\geq k$ and
									\[	\cS_{k}(\Phi)=t_{0}(\Phi)+\ldots+t_{k-1}(\Phi).	\]
			\end{enumerate}
\end{thm}
\begin{proof}			
			Let
			\begin{align*}
						\kappa(\Phi)\Mydef\inf\,\{\; j\geq 0&: \mbox{there exists a } \Psi\in\cA_{j}^{n}\mbox{ such that }\\
												&\int_{\T}\trace(\Phi(\zeta)\Psi(\zeta))\dm(\zeta)=t_{0}(\Phi)+\ldots+t_{n-1}(\Phi)\,\}
			\end{align*}
			Clearly, $\kappa(\Phi)$ may be infinite for arbitrary $\Phi$.
						
			
			Suppose $\kappa=\kappa(\Phi)$ is finite.  Then Lemma \ref{suffCondLemma} implies that $\Phi$ has a 
			$\kappa$-extremal function, $\cS_{\kappa}(\Phi)=\cS_{n}(\Phi)$, and the zero matrix function 
			is a best approximant to $\Phi$ under the $L^{\infty}(\bS_{1}^{n})$-norm.  
			In particular, $\Phi$ has order $k\leq\kappa(\Phi)$, $t_{j}(\Phi)=0$ for $j\geq k$, and 
			\[	\cS_{k}(\Phi)=t_{0}(\Phi)+\ldots+t_{k-1}(\Phi),	\]
			by Corollary \ref{mainCor}.
			
			On the other hand, if $\Phi$ has order $k$, $t_{j}(\Phi)=0$ for $j\geq k$, and 
			\[	\cS_{k}(\Phi)=t_{0}(\Phi)+\ldots+t_{k-1}(\Phi),	\]
			then $\Phi$ has a $k$-extremal function $\Psi\in\cA_{k}^{n}$ such that
			\[	\int_{\T}\trace(\Phi(\zeta)\Psi(\zeta))\dm(\zeta)=\cS_{k}(\Phi)=t_{0}(\Phi)+\ldots+t_{k-1}(\Phi).	\]
			Since $t_{j}(\Phi)=0$ for $j\geq k$, it follows that
			\[	\int_{\T}\trace(\Phi(\zeta)\Psi(\zeta))\dm(\zeta)=t_{0}(\Phi)+\ldots+t_{n-1}(\Phi).	\]
			Thus $\kappa(\Phi)\leq k$.		
			
			Hence, if either $\kappa(\Phi)$ is finite or $\Phi$ satisfies \emph{2}, then $k=\kappa(\Phi)$.
\end{proof}

We end this section by illustrating existence of very badly approximable matrix functions of order $k$ by giving
two simple examples; a $2\times 2$ matrix function of order 2 and a $3\times 3$ matrix function of order 2.

\begin{ex}\label{simpleEx1}
			Let
			\begin{equation*}
						\Phi=	\frac{1}{\sqrt{2}}
									\left(	\begin{array}{cc}
																1	&	-1\\
																1	&	1
													\end{array}\right)
									\left(	\begin{array}{cc}
																\bar{z}^{2}	&	\Zero\\
																\Zero				&	\bar{z}
													\end{array}\right).
			\end{equation*}
			It is easy to see that $\Phi$ is a continuous (and hence admissible) unitary-valued very badly approximable 
			matrix function with superoptimal singular values $t_{0}(\Phi)=t_{1}(\Phi)=1$.  We claim that $\Phi$ has order 
			2.  Indeed, the matrix function
			\begin{equation*}
						\Psi=\left(	\begin{array}{cc}
															z^{2}	&	\Zero\\
															\Zero	&	z\\
												\end{array}\right)\frac{1}{\sqrt{2}}
									\left(	\begin{array}{cc}
																1	&	1\\
																-1	&	1
													\end{array}\right)
			\end{equation*}
			satisfies
			\[	\int_{\T}\trace(\Phi(\zeta)\Psi(\zeta))\dm(\zeta)=2,	\]
			and so $\Phi$ has order 2 by Corollary \ref{unitaryCor}.
\end{ex}

\begin{ex}\label{simpleEx2}
			Let $t_{0}$ and $t_{1}$ be two positive numbers satisfying $t_{0}\geq t_{1}$.  Let
			\begin{equation*}
						\Phi=\left(	\begin{array}{ccc}
															t_{0}\bar{z}^{a}	&	\Zero							&	\Zero\\
															\Zero							&	t_{1}\bar{z}^{b}	&	\Zero\\
															\Zero							&	\Zero							&	\Zero
												\end{array}\right)
			\end{equation*}
			where $a$ and $b$ are positive integers.  It is easy to see that $\Phi$ is a continuous (and hence admissible) 
			very badly approximable matrix function with superoptimal singular values $t_{0}(\Phi)=t_{0}$, $t_{1}(\Phi)=t_{1}$, 
			and $t_{2}(\Phi)=0$.  Again, we have that $\Phi$ has order 2. After all, the matrix function
			\begin{equation*}
						\Psi=\left(	\begin{array}{ccc}
															z^{a}	&	\Zero	&	\Zero\\
															\Zero	&	z^{b}	&	\Zero\\
															\Zero	&	\Zero	&	\Zero
												\end{array}\right)
			\end{equation*}
			satisfies
			\[	\int_{\T}\trace(\Phi(\zeta)\Psi(\zeta))\dm(\zeta)=t_{0}+t_{1}=t_{0}(\Phi)+t_{1}(\Phi)+t_{2}(\Phi),	\]
			and so $\Phi$ has order 2 by Theorem \ref{suffCondition}, since $t_{1}(\Phi)=t_{1}>0$.
\end{ex}

{\centering \section{Unitary-valued very badly approximable matrix functions}\label{unitarySection}}

We lastly consider the class $\cU_{n}$ of admissible very badly approximable unitary-valued matrix functions
of size $n\times n$ and provide a representation of any $n$-extremal function $\Psi$ for a function $U\in\cU_{n}$ 
such that
\begin{equation}\label{unitPsi}
			\int_{\T}\trace(U(\zeta)\Psi(\zeta))\dm(\zeta)=t_{0}(U)+\ldots+t_{n-1}(U)
\end{equation}
holds.  Note that for any such $U$ we have that $t_{j}(U)=1$ for $0\leq j\leq n-1$.

When studying functions in $\cU_{n}$, it turns out that Toeplitz operators on Hardy spaces are quite useful.
For a matrix function $\Phi\in L^{\infty}(\M_{m,n})$, we define the \emph{Toeplitz operator} $T_{\Phi}$ by
\[	T_{\Phi}f=\pP_{+}\Phi f,\,\mbox{ for }f\in H^{2}(\C^{n}),	\]
where $\pP_{+}$ denotes the orthogonal projection from $L^{2}(\C^{n})$ onto $H^{2}(\C^{n})$.

It is well-known that, for any function $U\in\cU_{n}$, the Toeplitz operator $T_{U}$ is Fredholm and $\ind T_{U}>0$.  
(As usual, for a Fredholm operator $T$, its index, $\ind T$, is defined by $\dim\ker T-\dim\ker T^{*}$.)  In particular,
the Toeplitz operator $T_{\det U}$ is Fredholm and
\[	\ind T_{\det U}=\ind T_{U}. 	\]
This latter fact can be easily deduced by considering any \emph{thematic factorization} of $U$.  We refer the reader 
to Chapter 14 in \cite{Pe1} for more information concerning functions in $\cU_{n}$ and thematic factorizations.

In order to state the main result of this section, we first discuss the notion of Blaschke-Potapov products.  A matrix 
function $B\in H^{\infty}(\M_{n})$ is called a finite \emph{Blaschke-Potapov product} if it admits a factorization of the form
\[	B=U B_{1}B_{2}\ldots B_{m},	\]
where $U$ is a unitary matrix and, for each $1\leq j\leq m$,
\[	B_{j}=\frac{z-\lambda_{j}}{1-\bar{\lambda}_{j}z}P_{j}+(I-P_{j})	\]
for some $\lambda_{j}\in\D$ and orthogonal projection $P_{j}$ on $\C^{n}$.  The \emph{degree} of the Blaschke-Potapov 
product $B$ is defined to be
\[	\deg B\Mydef \sum_{j=1}^{m}\rank P_{j}.	\]
It turns out that every invariant subspace $\sL$ of multiplication by $z$ on $H^{2}(\C^{n})$ of finite codimension is of 
the form $B H^{2}(\C^{n})$ for some Blaschke-Potapov product of finite degree $\codim \sL$.  A proof of this fact may be 
found in Lemma 2.5.1 of \cite{Pe1}.

We now state the main result.
\begin{thm}\label{unitaryThm}
			Suppose $U\in\cU_{n}$ has an $n$-extremal function $\Psi$ such that $(\ref{unitPsi})$ holds.  Then $\Psi$ admits
			a representation of the form
			\[	\Psi=zh^{2}\Theta,	\]
			where $h\in H^{2}$ is an outer function such that $\|h\|_{L^{2}}=1$ and $\Theta$ is a finite Blaschke-Potapov
			product.  Moreover, the scalar functions $\det (U\Theta)$ and $\trace(U\Theta)$ are admissible badly
			approximable functions that admit the factorizations
			\[	\det (U\Theta)=\bar{z}^{n}\frac{\bar{h}^{n}}{h^{n}}\,\mbox{ and }\;\trace(U\Theta)=n\bar{z}\frac{\bar{h}}{h}.	\]
\end{thm}
\begin{proof}\label{nExtPsi}
			It follows from $(\ref{unitPsi})$ that all inequalities in $(\ref{traceIneqTwo})$ are equalities and so
			\begin{equation}\label{revealEq}
						\trace(U(\zeta)\Psi(\zeta))=\|U(\zeta)\Psi(\zeta)\|_{\bS_{1}^{n}}=n\|\Psi(\zeta)\|_{\M_{n}}
			\end{equation}
			holds for a.e. $\zeta\in\T$.  Since $U$ is unitary-valued, then
			\[	\|U(\zeta)\Psi(\zeta)\|_{\bS_{1}^{n}}=\|\Psi(\zeta)\|_{\bS_{1}^{n}},	\]
			and so
			\[	\|\Psi(\zeta)\|_{\bS_{1}^{n}}=n\|\Psi(\zeta)\|_{\M_{n}}	\]
			must hold for a.e. $\zeta\in\T$.  Therefore
			\[	s_{j}(\Psi(\zeta))=\|\Psi(\zeta)\|_{\M_{n}}\mbox{ for a.e. }\zeta\in\T,\, 0\leq j\leq n-1.	\]
			By the Singular Value Decomposition Theorem for matrices (or, more generally, the Schmidt Decomposition Theorem),
			it follows that
			\begin{equation}\label{PsiRep}
						\Psi(\zeta)=\|\Psi(\zeta)\|_{\M_{n}}V(\zeta)\mbox{ for a.e. }\zeta\in\T,
			\end{equation}
			for some unitary-valued matrix function $V$.  Let $h\in H^{2}$ be an outer function such that
			\[	|h(\zeta)|=\|\Psi(\zeta)\|_{\M_{n}}^{1/2}\mbox{ on }\T.	\]
			Consider also the matrix function $\Xi\Mydef h^{-2}\Psi$.  It follows from $(\ref{PsiRep})$ that
			\[	(\Xi^{*}\Xi)(\zeta)=\frac{1}{|h(\zeta)|^{4}}(\Psi^{*}\Psi)(\zeta)=I_{n}\mbox{ for a.e. }\zeta\in\T,	\]
			and so $\Xi$ is an inner function.  Thus $\Psi$ admits the factorization
			\[	\Psi=z h^{2}\Theta	\]
			for some $n\times n$ unitary-valued inner function $\Theta$ and an outer function $h\in H^{2}$ such that
			$\|h\|_{L^{2}}=1$.
			
			Note that the first equality in $(\ref{revealEq})$ indicates that the scalar function $\varphi\Mydef\trace(U\Theta)$
			satisfies
			\[	zh^{2}\varphi=n|h|^{2}\,\mbox{ on }\T,	\]
			or equivalently
			\[	\varphi=n\bar{z}\frac{\bar{h}}{h}.	\]
			Moreover, $\|H_{U\Theta}\|_{\rm e}\leq\|H_{U}\|_{\rm e}<1$, hence $\|H_{\varphi}\|_{\rm e}<n=\|H_{\varphi}\|$
			implying that $\varphi$ is an admissible badly approximable scalar function on $\T$.  We conclude that the Toeplitz 
			operator $T_{\varphi}$ is Fredholm and $\ind T_{\varphi}>0$ by the following well-known fact (c.f. Theorem 7.5.5 in \cite{Pe1}.)
			
			\textbf{Fact.}  \emph{Let $\varphi\in L^{\infty}$ be admissible.  Then $\varphi$ is badly approximable (i.e. the
			zero scalar function is a best approximant) if and only if $\varphi$ has constant modulus, the Toeplitz operator
			$T_{\varphi}$ is Fredholm, and $\ind T_{\varphi}>0$.}			 
			
			Returning to $(\ref{revealEq})$, it also follows that each eigenvalue of $U(\zeta)\Psi(\zeta)$ equals
			$\|\Psi(\zeta)\|_{\M_{n}}=|h(\zeta)|^{2}$ for a.e. $\zeta\in\T$ .  In particular, 
			\[	|h(\zeta)|^{2n}=\det U(\zeta)\Psi(\zeta)=(z^{n}h^{2n})(\zeta)\cdot\det U(\zeta)\cdot\det\Theta(\zeta)	\]
			holds a.e. $\zeta\in\T$.  By setting
			\[	\theta\Mydef\det\Theta\,\mbox{ and }\;u\Mydef\det U,	\]
			we have that $u$ admits the factorization
			\[	u=\bar{\theta}\bar{z}^{n}\frac{\bar{h}^{n}}{h^{n}}=\bar{\theta}\omega^{n},	\]
			where $\omega\Mydef\bar{z}\bar{h}/h=\varphi/n$.  Since the Toeplitz operator $T_{\omega}$ is Fredholm with
			positive index, $T_{u\bar{\omega}^{n}}$ is Fredholm as well.  Since $\ker T_{\theta}=\{\Zero\}$ and
			$u\bar{\omega}^{n}=\bar{\theta}$, then 
			\[	\dim(H^{2}\ominus\theta H^{2})=\dim\ker T_{\theta}^{*}=\dim\ker T_{\bar{\theta}}=\ind T_{\bar{\theta}}<\infty	\]
			and so $\theta$ is a finite Blaschke product.  The conclusion follows from the well-known lemma stated below.
\end{proof}

\begin{lemma}
		If $\Theta$ is a unitary-valued inner function such that $\det \Theta$ is a finite Blaschke product, then
		$\Theta$ is a Blaschke-Potapov product.
\end{lemma}

\begin{proof}
			Let $\theta=\det\Theta$.  It is easy to see that $\Theta^{*}\theta$ is an inner function.  Since $B\Mydef\theta I_{n}$ 
			is a finite Blaschke-Potapov product and $BH^{2}(\C^{n})\subset\Theta H^{2}(\C^{n})$, then $\Theta H^{2}(\C^{n})$ has 
			finite codimension, and so $\Theta$ must be a finite Blaschke-Potapov product.
\end{proof}

\begin{cor}\label{twoByTwo}
			Suppose $U\in \cU_{2}$ has a $2$-extremal function $\Psi$ such that $(\ref{unitPsi})$ holds.  If $U$ is a
			rational matrix function such that $\ind T_{U}=2$, then $\Theta$ is a unitary constant on $\T$.
\end{cor}
\begin{proof}
			Due to the results of \cite{PY}, $U$ admits a (thematic) factorization of the form
			\begin{equation*}
						U=
						\left(	\begin{array}{cc}
													\bar{w}_{1}	&	-w_{2}\\
													\bar{w}_{2}	&	w_{1}\\
										\end{array}\right)
						\left(	\begin{array}{cc}
													u_{0}	&	\Zero\\
													\Zero	&	u_{1}\\
										\end{array}\right)
						\left(	\begin{array}{cc}
													\bar{v}_{1}	&	\bar{v}_{2}\\
													-v_{2}			&	v_{1}\\
										\end{array}\right),
			\end{equation*}
			where $v_{1}, v_{2}, w_{1}$ and $w_{2}$ are scalar rational functions such that
			\[	|v_{1}|^{2}+|v_{2}|^{2}=|w_{1}|^{2}+|w_{2}|^{2}=1\mbox{ a.e. on }\T,	\]
			$v_{1}$ and $v_{2}$ have no common zeros in the unit disk $\D$, $w_{1}$ and $w_{2}$ have no common zeros in $\D$, 
			and $u_{0}$ and $u_{1}$ are scalar badly approximable rational unimodular functions on $\T$.  These 
			results may also be found in Sections 5 and 12 from Chapter 14 of \cite{Pe1}.
						
			Suppose $\Psi=zh^{2}\Theta$ is an $n$-extremal function for $U$ such that $(\ref{unitPsi})$ holds as in the
			conclusion of Theorem \ref{unitaryThm}.  Assume, for the sake of contradiction, that $\Theta$ is not a unitary
			constant.
			
			Since $u_{j}$ is a scalar badly approximable rational unimodular function on $\T$, it admits a factorization of the form
			\[	u_{j}=c_{j}\bar{z}^{k_{j}}\frac{\bar{h}_{j}}{h_{j}},	\]
			where $c_{j}$ is a unimodular constant, the function $h_{j}$ is $H^{\infty}$-invertible, and $k_{j}=\ind T_{u_{j}}$, for $j=0,1$.  
			In particular, we have
			\[	u\theta=c_{0}c_{1}\bar{z}^{2}\theta\frac{\bar{h}_{0}}{h_{0}}\frac{\bar{h}_{1}}{h_{1}},	\]
			as $k_{0}+k_{1}=\ind T_{U}=2$, where $\theta\Mydef\det\Theta$ and $u\Mydef\det U$.
			
			On the other hand, by Theorem \ref{unitaryThm}, 
			\[	u\theta=\bar{z}^{2}\frac{\bar{h}^{2}}{h^{2}}	\]
			and so the function $h^{2}h_{0}^{-1}h_{1}^{-1}$ and its conjugate
			\[	\frac{\bar{h}^{2}}{\bar{h}_{0}\bar{h}_{1}}=c_{0}c_{1}\theta\frac{h^{2}}{h_{0}h_{1}}	\]
			belong to $H^{1}$.  Therefore $h^{2}h_{0}^{-1}h_{1}^{-1}$ equals a constant and so $\theta$ equals a constant as well.
			Thus, the conclusion follows from the fact that $\theta \Theta^{*}$ is an inner function. 
\end{proof}

We end this section with an example to illustrate some of our main results.

\begin{ex}
Consider the matrix function
\begin{equation*}
			U	=\left(	\begin{array}{cc}
												\bar{z}	&	\Zero\\
												\Zero		&	\bar{z}									
									\end{array}\right)
						\frac{1}{\sqrt{2}}\left(	\begin{array}{cc}
												1		&	\bar{z}\\
												-z	&	1									
									\end{array}\right)
					=\frac{1}{\sqrt{2}}\left(	\begin{array}{cc}
																					\bar{z}	&	\bar{z}^{2}\\
																					-1			&	\bar{z}									
																		\end{array}\right).
\end{equation*}
Clearly, $U$ belongs to $\cU_{2}$ and it has superoptimal singular values $t_{0}(U)=t_{1}(U)=1$.

We ask the question, \emph{is there a $2$-extremal function $\Psi$ for $U$ such that $(\ref{unitPsi})$ holds with
$n=2$?}  Let us assume for the moment that such a function $\Psi$ exists.  In this case, Corollary \ref{twoByTwo} 
implies that $\Psi$ must be of the form $\Psi=zh^{2}\Theta$, where
\begin{equation*}
			\Theta=\left(	\begin{array}{cc}
													a	&	b\\
													c	&	d
										\end{array}\right)
\end{equation*}
is a unitary constant and $h$ is an outer function in $H^{2}$ such that $\|h\|_{L^{2}}=1$. Since
\[	\bar{z}^{2}\frac{\bar{h}^{2}}{h^{2}}=\det(U\Theta)=\bar{z}^{2}(ad-bc),	\]
it is easy to see that $h^{2}$ and its conjugate belong to $H^{1}$, and so $h^{2}$ is a constant of modulus 1.
Relabeling the scalars $a, b, c,$ and $d$, we may assume that $h^{2}$ equals 1 a.e. on $\T$.  Thus,
\[	2\bar{\zeta}=\trace(U(\zeta)\Theta(\zeta))=\frac{1}{\sqrt{2}}\left(a\bar{\zeta}+c\bar{\zeta}^{2}-b+d\bar{\zeta}\right)	\]
holds for a.e. $\zeta\in\T$, and so $b=c=0$ and $a+d=2\sqrt{2}$.  However, $\Theta$ is unitary valued so
it must be the case that $|a|=|d|=1$, and so
\[	2\sqrt{2}=a+d=|a+d|\leq|a|+|d|=2,	\]
which is a contradiction.  Thus no such $\Psi$ exists.  In particular, we must have that $\Phi$ does not have order 2 or
$\sigma_{2}(\Phi)<t_{0}(\Phi)+t_{1}(\Phi)=2$ by Theorem \ref{classification}.

Actually, we have already shown that \emph{the zero matrix function is not a best approximant to $U$ under the 
$L^{\infty}(\bS_{1}^{2})$ norm}, i.e. $\sigma_{2}(\Phi)<2$.  Indeed, we have 
\[	\dist_{L^{\infty}(\bS_{1}^{2})}(U, H^{\infty}(\M_{2}))<t_{0}(U)+t_{1}(U)=\|U\|_{L^{\infty}(\bS_{1}^{2})},	\]
by $(\ref{strictIneq})$.

We now ask, \emph{does $U$ have order 1, order 2, or is $U$ inaccessible?} It is clear that $U$ has a $1$-extremal function 
by Remark \ref{vectorCaseRem}.  In fact, it is easy to check that the matrix function
\begin{equation*}
			\Psi_{1}=\frac{z}{\sqrt{2}}\left(	\begin{array}{cc}
																							1	&	\Zero\\
																							z	&	\Zero
																				\end{array}\right)
\end{equation*}
defines a $1$-extremal function for $U$ and
\[	\cS_{1}(U)=\int_{\T}\trace(U(\zeta)\Psi_{1}(\zeta))\dm(\zeta)=\|H_{U}\|=t_{0}(U)=1.	\]
However, \emph{$U$ does not have order 1}.  Indeed, one can see that the matrix function 
\begin{equation*}
			\Psi_{*}=\frac{z}{\sqrt{3}}\left(	\begin{array}{cc}
																								1	&	\Zero\\
																								z	&	1
																				\end{array}\right)
\end{equation*}
belongs to $H^{1}_{0}(\M_{2})$, $\|\Psi_{*}\|_{L^{1}(\M_{2})}\leq 1$, and
\[	1<\sqrt{\frac{3}{2}}=\int_{\T}\trace(U(\zeta)\Psi(\zeta))\dm(\zeta)\leq\cS_{2}(U).	\]
Therefore, either $U$ has order 2 or $U$ is inaccessible.  This matter requires further investigation.
\end{ex}

\smallskip

\textbf{Acknowledgment.}  This article is based in part on the author's Ph.D. dissertation at Michigan State University.
Also, the author would like to thank Professor Vladimir V. Peller for communicating Theorem 
\ref{SInequality} and for suggesting corrections on earlier versions of this paper.

\end{document}